\documentclass[a4paper,12pt]{amsart}
\usepackage[utf8]{inputenc}
\usepackage[T1]{fontenc}
\usepackage[UKenglish]{babel}
\usepackage[margin=30mm]{geometry}

\usepackage{color}

\usepackage{marginnote}
\usepackage{graphicx}

\usepackage{amsmath,amssymb,amsfonts,amsthm}
\usepackage{mathrsfs,eucal,dsfont}
\usepackage{verbatim,enumitem}

\usepackage{hyperref,url}

\renewcommand{\le}{\leqslant}
\renewcommand{\ge}{\geqslant}
\renewcommand{\leq}{\leqslant}
\renewcommand{\geq}{\geqslant}

\newcommand{\et}{\quad\text{and}\quad}

\newcommand{\real}{\mathds R}
\newcommand{\Pp}{\mathds P}
\newcommand{\Ee}{\mathds E}
\newcommand{\I}{\mathds 1}
\newcommand{\var}{\textmd{Var}}
\newcommand{\Bscr}{\mathscr{B}}
\newcommand{\rd}{{\mathds R^d}}
\newcommand{\id}{\operatorname{id}}

\newcommand{\scalp}[2]{\langle #1,\,#2\rangle}

\theoremstyle{plain}
\newtheorem{theorem}{Theorem}[section]
\newtheorem{lemma}[theorem]{Lemma}
\newtheorem{proposition}[theorem]{Proposition}

\theoremstyle{definition}

\newtheorem{example}[theorem]{Example}
\newtheorem{remark}[theorem]{Remark}
\newtheorem*{ack}{Acknowledgement}
\newtheorem*{notation}{Notation}

\begin{document}
\allowdisplaybreaks

\title[A Unified Approach to Coupling SDEs with L\'{e}vy Noise]{A Unified Approach to Coupling SDEs driven by L\'{e}vy Noise and some Applications}

\author[M.~Liang]{Mingjie Liang}
\address[M.~Liang]{%
College of Information Engineering, Sanming
University, 365004 Sanming, P.R. China}
\email{liangmingjie@aliyun.com}

\author[R.L.~Schilling]{Ren\'{e} L.\ Schilling}
\address[R.\ Schilling]{%
TU Dresden, Fakult\"{a}t Mathematik, Institut f\"{u}r Mathematische Stochastik, 01062 Dresden, Germany}
\email{rene.schilling@tu-dresden.de}

\author[J.~Wang]{Jian Wang}
\address[J.~Wang]{%
College of Mathematics and Informatics \& Fujian Key Laboratory of Mathematical Analysis and Applications (FJKLMAA), Fujian Normal University, 350007 Fuzhou, P.R. China}
\email{jianwang@fjnu.edu.cn}

\keywords{L\'{e}vy process; coupling operator; coupling by reflection; refined basic coupling; optimal coupling.}
\subjclass[2010]{60J35; 60G51; 60H10; 60J25; 60J75.}

\maketitle

\begin{abstract}
    We present a general method to construct couplings of stochastic differential equations driven by L\'{e}vy noise in terms of coupling operators. This approach covers both coupling by reflection and refined basic coupling which are often discussed in the literature. As an application, we establish regularity results for the transition semigroups of the solutions to stochastic differential equations driven by additive L\'{e}vy noise.
\end{abstract}

Coupling is a well-known powerful tool in the study of Markov processes, see \cite{Ligg,Lin,Tho}. It has been efficiently used to show regularity properties of Markov semigroups and ergodicity of Markov processes. There are many publications on coupling of diffusion processes, see for instance \cite{LR, CL, PW, HS, BK} and the references therein, but only few papers consider the coupling of jump processes. The first systematic investigations on coupling of L\'evy processes are \cite{SW, BSW, SSW} and \cite[Chapter 6.2]{BSW1}, but -- compared to the diffusion case -- the theory is still in its infancy.

In this paper, we consider $d$-dimensional stochastic differential equations (SDEs) driven by additive pure jump L\'evy noise
\begin{equation}\label{s1}
    dX_t
    = b(X_{t})\,dt+ dZ_t,\quad X_0=x\in \rd,
\end{equation}
where $b:\rd\to\rd$ is a measurable function and $Z=(Z_t)_{t\geq 0}$ is a pure jump L\'{e}vy process on $\rd$. We assume that the SDE \eqref{s1} has a unique strong solution $X=(X_t)_{t\geq 0}$. This holds, for example, if $b$ satisfies the local Lipschitz and linear growth conditions, see \cite[Chapter~IV.9]{ike-wat}, or if $b$ is H\"{o}lder continuous and $Z$ a L\'evy process satisfying some moment condition for the L\'evy measure  at zero and at infinite and  such that its transition semigroups enjoy certain regularity properties, see e.g.\ \cite{CSZ, kuehn-rs, Po1, Zhang1}. It is easy to see that the generator of $X$ is given by \begin{equation}\label{e:ope}
    Lf(x)
    = \scalp{\nabla f(x)}{b(x)} + \int_\rd \left[f(x+u)-f(x)-\scalp{\nabla f(x)}{u}\I_{(0,1)}(|u|)\right]\nu(du),
\end{equation}
where $\nu$ is the L\'evy measure of the pure jump L\'{e}vy process $Z$.

We have two aims in mind: First, we want to find a uniform formulation for coupling of the SDE \eqref{s1} -- this serves as model case for more general SDEs with multiplicative noise, see Section~\ref{subsec-multi}; this is done using the concept of coupling operators covering all currently known couplings for L\'evy processes. The other aim is to establish new regularity results for the transition semigroups of the solution of the SDE \eqref{s1} -- and, in particular, for L\'evy processes -- illustrating the power of the coupling and coupling operator method when applied to L\'evy processes.

\begin{notation}
    Most of our notation is standard or self-explanatory. L\'evy measures $\nu(du)$ and L\'evy kernels $\nu(x,du)$ are, as usual, defined on $\rd\setminus\{0\}$; for simplicity we will not make this explicit in our notation and keep writing $\int_{\rd}\dots\nu(du)$ etc. By $a\wedge b$ we denote the minimum of $a$ and $b$, and agree that ``$\wedge$'', when combined with ``$+$'' or ``$-$'', takes precedence over these operations, i.e.\ $a\pm a\wedge b = a \pm (a\wedge b)$.
\end{notation}

\section{Coupling operators for SDEs with additive L\'evy noise}\label{section1}

Let $L$ be a linear operator from $C_b^2(\rd)$ to $B(\real)$. Recall that the tensor product of two functions $f,g:\rd\to\real$ is the function $f\otimes g(x,y):= f(x)g(y)$.  Following \cite[Chapter 2.1]{Chen} we call a linear operator $\widetilde L:C_b^{2}(\real^{2d})\to B(\real)$ a \emph{coupling operator} with marginal $L$, if
\begin{align*}
    \widetilde L (f\otimes \I)(x,y) &= Lf(x) &&\text{for all $x,y\in \rd$ and $f\in C^2_b(\rd)$},\\
    \widetilde L (\I\otimes g)(x,y) &= Lg(y) &&\text{for all $x,y\in \rd$ and $g\in C^2_b(\rd)$}.
\end{align*}
Typically, $L$ will be the infinitesimal generator of a Markov process. The main purpose of this section is to study a general formula for the coupling operator $\widetilde Lf$.

Assume, for a moment, that $L$ is the generator of a Feller process $X=(X_t)_{t\geq 0}$ such that the test functions $C_c^\infty(\rd)$ are contained in the domain $D(L)$. It is well-known, cf.\ \cite{BSW1}, that $\overline{C_c^\infty(\rd)}^{\|\cdot\|_\infty}\subset D(L)$ and $Lf$, $f\in C_c^\infty(\rd)$, is necessarily of the form
\begin{align*}
     Lf(x)
     &= \scalp{\nabla f(x)}{b(x)} + \frac 12\mathop{\mathrm{div}} Q(x)\nabla f(x)\\
     &\qquad\mbox{}+ \int_\rd \left[f(x+u)-f(x)-\scalp{\nabla f(x)}{u}\I_{(0,1)}(|u|)\right]\nu(x,du);
\end{align*}
here, $(b(x),Q(x),\nu(x,du))$ is for every fixed $x\in\rd$ a L\'evy triplet, i.e.\ $b(x)\in\rd$, $Q(x)\in\real^{d\times d}$ is positive semidefinite, $\int_\rd [1\wedge |u|^2]\,\nu(x,du)<\infty$ and all expressions are measurable and locally bounded in $x$.

Therefore, the following Ansatz  provides a natural candidate for a coupling operator related to \eqref{e:ope}: for any $f\in C_b^2(\real^{2d})$,
\begin{equation}\label{e:gencou}
\begin{split}
    \widetilde L f (x,y)
    &=\scalp{\nabla_xf(x,y)}{b(x)} + \scalp{\nabla_yf(x,y)}{b(y)}\\
    &\qquad\mbox{} + \int_{\rd\times \rd} \Big[f(x+u,y+v)-f(x,y)-\scalp{\nabla_x f(x,y)}{u} \I_{(0,1)}(|u|)\\
    &\qquad\qquad\qquad\mbox{  } -\scalp{\nabla_y f(x,y)}{v} \I_{(0,1)}(|v|)\Big]\,\widetilde\nu(x,y,du,dv),
\end{split}\end{equation}
where $\nabla_xf(x,y)$ and $\nabla_yf(x,y)$ denote the gradient of $f(x,y)$ with respect to $x$ and $y$, and $\widetilde\nu(x,y,du,dv)$ is a L\'evy-type kernel, i.e.\ a measure on $\real^{2d}\setminus\{0\}$ satisfying
\begin{equation}\label{e:note}
    \int_{\real^{2d}} \left[1\wedge (|u|^2+|v|^2)\right] \widetilde\nu(x,y,du,dv)<\infty, \quad x,y\in\rd.
\end{equation}

\begin{lemma}\label{L:lem}
    The operator $\widetilde L$ defined by \eqref{e:gencou} is a coupling operator with marginal operator $L$ of the form \eqref{e:ope}, if and only if, $\widetilde\nu(x,y,du,dv)$ satisfies for all $A,B\in\Bscr(\rd\setminus\{0\})$ and $x,y\in\rd$ the following conditions
    \begin{equation}\label{e:p1}
        \widetilde\nu(x,y,A\times \rd)= \nu(A), \quad
        \widetilde\nu(x,y,\rd\times B)= \nu(B).
    \end{equation}
\end{lemma}
\begin{proof}
By \eqref{e:gencou} we have for any $f\in C_b^2(\rd)$,
\begin{align*}
    \widetilde L (f\otimes \I)(x,y)
    &= \scalp{\nabla f(x)}{b(x)}\\
    &\quad\mbox{} + \int_{\rd\times \rd} \left[f(x+u)-f(x)-\scalp{\nabla  f(x)}{u} \I_{(0,1)}(|u|)\right] \widetilde\nu(x,y,du,dv).
\end{align*}
Let $f\in C_c^2(\rd\setminus\{0\})$. We have
\begin{gather*}
    \widetilde L (f\otimes\I)(0,y)
    = \int_{\rd\times \rd} f(u)\,\widetilde\nu(x,y,du,dv)
    \et
    Lf(0) = \int_\rd f(u)\,\nu(du).
\end{gather*}
Since $\widetilde L (f\otimes \I) = Lf$, we get the first equality in \eqref{e:p1} since the family $C_c^2(\rd\setminus\{0\})$ is measure-determining on $\rd\setminus\{0\}$. The second equality follows in a similar way.

Let us show that \eqref{e:p1} is also sufficient. For any $f\in C_b^2(\rd)$ and $x\in \rd$, set $F_x(u):=f(x+u)-f(x)-\scalp{\nabla f(x)}{u} \I_{(0,1)}(|u|)$. By definition, $F_x(u)\in C_b(\rd)$ with $F_x(0)=0$. Thus, \eqref{e:p1} along with a standard approximation argument yields
\begin{gather*}
    \int_{\rd\times \rd} F_x(u)\,\widetilde \nu(x,y,du,dv)
    = \int_{\rd} F_x(u)\,\nu(du).
\intertext{Similarly, we get in the other coordinate direction}
    \int_{\rd\times \rd} F_y(v)\,\widetilde \nu(x,y,du,dv)
    = \int_{\rd} F_y(v)\,\nu(dv).
\end{gather*}
Hence, $\widetilde L$ defined by \eqref{e:gencou} is a coupling operator with marginal operator $L$.
\end{proof}

The condition \eqref{e:p1} for a coupling operator is stronger than the requirement \eqref{e:note} for general L\'evy-type operators. This means that the class of L\'evy-type coupling operators is smaller than the class of L\'evy-type operators -- but to-date we are not aware of a structural characterization of general L\'evy-type coupling operators, and we have to restrict ourselves to concrete examples.

To proceed, we need some further notation. For any bi-measurable function $f:\rd\to \rd$ and $A\in \Bscr(\rd)$, we define
\begin{gather*}
    (\nu \circ f)(A)=\nu(f(A))
    \et
    \mu_{\nu,f}=\nu\wedge (\nu \circ f);
\end{gather*}
the minimum of two measures $\nu_1$ and $\nu_2$ on $(\rd,\Bscr(\rd))$ is defined as $\nu_1\wedge \nu_2=\nu_1-(\nu_1-\nu_2)^+$ where $(\nu_1-\nu_2)^{\pm}$ are the positive resp.\ negative parts from the Hahn-Jordan decomposition of the signed measure $\nu_1-\nu_2$. For any $1\leq i<n+1\leq \infty$, let $\nu_i$ be a nonnegative measure on $(\rd,\Bscr(\rd))$ such that $\sum_{i=1}^n \nu_i\leq \nu$, and $\Psi_i:\rd\to \rd$ a bijective and continuous mapping, i.e.\ $\Psi_i$ is invertible and continuous satisfying $\Psi_i(\rd)=\rd$. In particular, $\Psi_i$ is bi-measurable from $\rd$ to $\rd$. For any $f\in C_b^2(\real^{d}\times\rd)$ and $x,y\in \rd$, we set
\begin{equation}\label{cp1}\begin{split}
    \widetilde{L} f(x,y)
    &=\scalp{\nabla_xf(x,y)}{b(x)} + \scalp{\nabla_yf(x,y)}{b(y)}\\
    &\quad\mbox{} + \sum_{i=1}^n\int_{\real^{d}}\Big[ f(x+z,y+ \Psi_i(z))-f(x,y)-\scalp{\nabla_xf(x,y)}{z} \I_{(0,1)}(|z|)\\
    &\quad\qquad\quad\mbox{} -\scalp{\nabla_yf(x,y)}{\Psi_i(z)} \I_{(0,1)}(|\Psi_i(z)|)\Big]\,\mu_{\nu_i,\Psi_i}(dz)\\
    &\quad\mbox{}+\int_{\real^{d}}\Big[f(x+z,y+z)-f(x,y)-\scalp{\nabla_xf(x,y)}{z} \I_{(0,1)}(|z|)\\
    &\quad\qquad\quad\mbox{}-\scalp{\nabla_yf(x,y)}{z} \I_{(0,1)}(|z|)\Big] \Big(\nu -\sum_{i=1}^n\mu_{\nu_i,\Psi_i}\Big)(dz).
\end{split}\end{equation}

\begin{proposition}\label{P:co} If
\begin{equation}\label{e:cou}
    \sum_{i=1}^n\mu_{\nu_i,\Psi_i}=\sum_{i=1}^n\mu_{\nu_i,\Psi_i^{-1}},
\end{equation}
then, the operator $ \widetilde{L}$ defined by \eqref{cp1} is a coupling operator with marginal operator $L$ given by \eqref{e:ope}.
\end{proposition}
\begin{proof}
Set
\begin{gather*}
    \widetilde \nu(x,y,du,dv)
    = \sum_{i=1}^n \mu_{\nu_i,\Psi_i}(du)\,\delta_{\Psi_i(u)}(dv)
    + \Big(\nu -\sum_{i=1}^n\mu_{\nu_i,\Psi_i}\Big)(du)\,\delta_{u}(dv).
\end{gather*}
The operator $\widetilde{L}$ defined by \eqref{cp1} is of the form \eqref{e:gencou} with the L\'evy type kernel $\widetilde \nu(x,y,du,dv)$ shown above. It is clear that we have $\widetilde \nu(x,y,A\times \rd)=\nu(A)$ for any $x, y\in \rd$ and $A\in \Bscr(\rd\setminus\{0\})$.  On the other hand, we have
\begin{gather*}
    (\mu_{\nu_i,\Psi_i}\circ\Psi_i^{-1})(A)
    = \mu_{\nu_i,\Psi_i^{-1}}(A)
    \quad\text{for all $A\in\Bscr(\rd\setminus\{0\})$ and $1\leq i< n+1$}.
\end{gather*}
Together with \eqref{e:cou} this yields that for $x, y\in \rd$ and $B\in \Bscr(\rd\setminus\{0\})$, $\widetilde \nu(x,y, \rd\times B)=\nu(B)$. The claim follows from Lemma~\ref{L:lem}.
\end{proof}

The coupling operator $\widetilde{L}$ defined by \eqref{cp1} can be uniquely described by the drift $b(x)$ and the following jump system
\begin{equation}\label{basic-coup-4}
(x,y)\longmapsto
    \begin{cases}
    (x+z, y+\Psi_i(z)), & \mu_{\nu_i,\Psi_i}(dz) \text{\ \ for\ \ } 1\leq i< n+1;\\
    (x+z, y+z), & \big(\nu -\sum_{i=1}^n \mu_{\nu_i,\Psi_i} \big)(dz).
    \end{cases}
\end{equation}
We will adopt this description throughout the rest of the paper.

\begin{remark}
    In most applications one needs a pathwise realization of the coupling in form of a Markov process, that is a $2d$-dimensional Markov process $(X_t,Y_t)_{t\geq 0}$ such that $(X_t-X_0)_{t\geq 0}$ and $(Y_t-Y_0)_{t\geq 0}$ are Markov processes with infinitesimal generator $L$. Clearly, if $(X_t,Y_t)_{t\geq 0}$ exists, then the coupling operator $\widetilde L$ (with the generator $L$ as marginal operator) is indeed the infinitesimal generator of $(X_t,Y_t)_{t\geq 0}$. The converse is more of a problem: from the mere definition of a coupling operator $\widetilde L$ we cannot immediately deduce the existence of an associated Markov process -- we refer to \cite{BSW1} for an exhaustive discussion of the existence of processes generated by L\'evy-type operators.

    In general, one needs a further argument to deduce the existence of a coupling process $(X_t,Y_t)_{t\geq 0}$.
    For diffusions, the well-posedness of the associated martingale problem is the method of choice, see \cite[Sections 2 and 3]{CL} and \cite[Section 2]{PW}, see also \cite[Section 3.1]{Wang10} and \cite[Section 2.2]{Wang} for the L\'evy case.

    In the present context, all processes are given by SDEs, so it is more natural to require the existence of a strong solution to the SDE, see e.g.~\cite{Lwcw, LW18}.
\end{remark}

\section{Explicit coupling processes for SDEs with additive L\'{e}vy noise via coupling operators}\label{section2}
In this section, we will establish three kinds of coupling processes for the SDE \eqref{s1} by making full use of the coupling operator constructed in the previous section. In the literature, these three -- in general highly non-trivial -- couplings are treated in different settings; it is, therefore, surprising that we can handle them in a unified framework based on the coupling operator \eqref{cp1}.

\subsection{Coupling by reflection: rotationally symmetric L\'{e}vy noise}\label{subse1}
Assume that $Z=(Z_t)_{t\geq 0}$ is a pure jump rotationally  symmetric  L\'evy process with L\'evy measure by $\nu$.  For any $x,y,z\in\rd$, we write
\begin{equation}\label{e:effrr}
    R_{x,y}(z)
    :=
    \begin{cases}
        \displaystyle z-\frac{2\scalp{x-y}{z}}{|x-y|^2}(x-y),&\text{if\ \ } x\neq y,\\[\medskipamount]
        \displaystyle z,&\text{if\ \ } x=y
    \end{cases}
\end{equation}
for the reflection at the hyperplane orthogonal to $x-y$. Obviously, $R_{x,y}(z) = R_{y,x}(z)$, $R_{x,y}(z) = R_{x,y}^{-1}(z)$, $|R_{x,y}(z)|=|z|$ and $R_{x,y}(z)-(x-y) = R_{x,y}(z)+ R_{x,y}(x-y) = R_{x,y}(z+x-y)$.

Setting $n=1$, $\Psi_1(z)=R_{x,y}(z)$ and $\nu_1(dz)=\I_{\{|z|< \eta|x-y|\}}\,\nu(dz)$ for some fixed $\eta\in(0,\infty]$, \eqref{basic-coup-4} becomes
\begin{equation}\label{basic-coup-5}
  (x,y)\longmapsto
  \begin{cases}
    (x+z, y+R_{x,y}(z)), & \I_{\{|z|\leq \eta|x-y|\}}\,\nu(dz);\\
    (x+z, y+z), & \I_{\{|z|> \eta|x-y|\}}\,\nu(dz).
  \end{cases}
\end{equation}
Since $\nu$ is rotationally symmetric, $\nu_1$ is invariant under the transformation $R_{x,y}(z)\rightsquigarrow z$, as $R_{x,y}(z)=R_{x,y}^{-1}(z)$ and $|R_{x,y}(z)|=|z|$. This shows $\nu_1\circ\Psi_1 = \nu_1\circ\Psi^{-1}_1 = \nu_1$ which means that \eqref{e:cou} is satisfied. Thus, according to Proposition~\ref{P:co}, the jump system \eqref{basic-coup-5} determines a coupling operator $\widetilde L$.

Let us briefly verify the existence of a $2d$-dimensional coupling process which is generated by the coupling operator $\widetilde L$ given by \eqref{basic-coup-5}. By the L\'evy-It\^{o} decomposition, there exists a Poisson random measure $N(dt,dz)$ such that
\begin{gather*}
    dZ_t = \int_{\{|z|\geq 1\}}z\, N(dt,dz)+\int_{\{|z|<1\}} z\,\widetilde N(dt,dz),
\end{gather*}
where $\widetilde N(dt,dz)=N(dt,dz)-dt\,\nu(dz)$ is the compensated Poisson random measure. To keep notation simple, we set $\breve{N}(dt,dz)=\I_{(0,1)}(|z|)\, \widetilde N(dt,dz) + \I_{[1,\infty)}(|z|)\,N(dt,dz)$, and so
\begin{gather*}
    dZ_t=\int z\,\breve{N}(dt,dz).
\end{gather*}
Consider the following system of SDEs on $\real^{2d}$:
\begin{equation}\label{e:sde1}
\left\{\begin{aligned}
    dX_t &= b(X_t)\,dt+ \int z\,\breve{N}(dt,dz), &&t>0,\\
    dY_t &= b(Y_t)\,dt + \int_{\{|z|< \eta |X_t-Y_t|\}}R_{X_{t-},Y_{t-}}(z)\,\breve{N}(dt,dz)\\
            &\phantom{= b(Y_t)\,dt}\mbox{} + \int_{\{|z|\geq \eta |X_t-Y_t|\}}z\,\breve{N}(dt,dz), && t>0.
\end{aligned}\right.
\end{equation}
 For any $x,y,z\in \rd$ with $x\neq y$ we have
\begin{gather*}
    \left(\id_d-\frac{2}{|x-y|^2}{(x-y)(x-y)^\top}\right) z
    = R_{x,y}(z),
\end{gather*} where $\id_d$ denotes the $d\times d$ identity matrix.
Thus, for any fixed $z\in \rd$, $(x,y)\mapsto R_{x,y}(z)$ is locally Lipschitz continuous on $\{(x,y)\in \real^{2d}:x\neq y\}$. If we assume, in addition, that the drift term $b$ is Lipschitz continuous, then the SDE \eqref{e:sde1} has a unique strong solution $(X_t,Y_t)$ up to $t<\tau$, where $\tau$ is the coupling time defined by
\begin{equation}\label{e:refl4}
    \tau:=\inf\{t>0:X_t=Y_t\},
\end{equation}
see also the discussion in the proof of Proposition~\ref{P:strong} below.  Since, by assumption, the first SDE in \eqref{e:sde1} has a unique strong solution $(X_t)_{t\geq 0}$, it is natural to identify the solution of \eqref{e:sde1} with $(X_t,X_t)$ for all $t\geq \tau$. By It\^{o}'s formula, we can easily verify that the generator of $(X_t,Y_t)_{t\geq 0}$  is the coupling operator $\widetilde L$ given by \eqref{basic-coup-5}.

Recall that coupling by reflection for SDEs driven by an additive Brownian motion $B=(B_t)_{t\geq 0}$ is realized through the following system of SDEs, cf.~\cite{LR, CL}:
\begin{gather*}\Bigg\{\begin{aligned}
    dX_t &= b(X_t)\,dt+ dB_t,\\
    dY_t &= b(Y_t)\,dt + (\id_d-2e_te_t^\top)\,dB_t,
\end{aligned}  \end{gather*}
where
\begin{gather*}
    e_t:=|X_t-Y_t|^{-1}(X_t-Y_t),
\end{gather*}
$e_t^\top$ is the transpose of $e_t$, and $\tau$ is defined as \eqref{e:refl4}. In particular, we have $(X_t,Y_t) = (X_t,Y_t)\I_{\{t<\tau\}}+(X_t,X_t)\I_{\{t\geq \tau\}}$.  For any $t>0$ the matrix $A_t=\id_d-2e_te_t^\top$ is an orthogonal matrix and, by the L\'evy characterization of Brownian motion, the process ${B}^{\#}$ defined by ${B}^{\#}_t := A_tB_t$, $t>0$, is also a Brownian motion. We can use a similar idea to construct the corresponding coupling for the SDE \eqref{s1}: If $Z=(Z_t)_{t\geq 0}$ is a rotationally symmetric pure jump L\'evy process, then the process $Z^{\#}$ defined by $Z^{\#}_t:=A_tZ_t$, $t>0$, is again a rotationally symmetric pure jump L\'evy process which has the same distribution as $(Z_t)_{t\geq 0}$.
Indeed, let $L^{\#}$ be the generator of the process $Z^{\#}$. For any $f\in C_b^2(\rd)$, we know
\begin{align*}
    L^{\#}f(x)
    &= \int\left(f(x+R_{x,y}(z))-f(x)-\scalp{\nabla f(x)}{R_{x,y}(z)} \I_{(0,1)}(|R_{x,y}(z)|)\right)\nu(dz)\\
    &= \int\left(f(x+z)-f(x)-\scalp{\nabla f(x)}{z} \I_{(0,1)}(|z|)\right)\nu(dz)
    = Lf(x),
\end{align*}
where we use the fact that $\nu$ is invariant under the change of variables $R_{x,y}(z)\rightsquigarrow z$ due to the rotational symmetry of the process $Z$. This shows that $Z^{\#}$ is a pure-jump L\'evy process with the same L\'evy measure as $Z$, hence $Z$ and $Z^{\#}$ coincide in law. In particular, the associated coupling process can be realized using   the SDE \eqref{e:sde1} with $\eta=\infty$. This is the reason why we call \eqref{basic-coup-5} \emph{coupling by reflection}.

This construction is, in general, not always the best choice. In contrast to the diffusion case, the above construction allows for a situation that two jump processes -- even if they are already close -- suddenly jump far apart.  In order to apply coupling by reflection we have to choose the parameter $\eta$ carefully. We also emphasize that the above argument still works if $Z$ is of the form $Z = Z'+Z''$ where $Z'$, $Z''$ are independent L\'evy processes and $Z''$ is rotationally symmetric, see e.g.~\cite{Wang}.

\subsection{Refined basic coupling: general L\'evy noise}\label{subsection2}
The notion of \emph{basic coupling} was introduced by M.-F.\ Chen, \cite[Example 2.10]{Chen} when studying Markov $q$-processes. The underlying idea of this coupling is to force the two marginal processes to jump to the same point with the biggest possible rate.  In the L\'evy case, the biggest jump rate takes the form $\mu_{y-x}(dz):=[\nu\wedge (\delta_{y-x} \ast\nu)](dz)$,  where $\nu$ is the L\'evy measure and  $x\neq y$ are the positions of the two marginal processes immediately before the jump.

Let $\nu$ be the L\'evy measure of the L\'evy process $Z=(Z_t)_{t\geq 0}$. We stress that the construction of coupling in this section does not require any further (geometric) assumptions on the L\'evy measure. For any $\kappa>0$ and $x,y\in\rd$, we define
\begin{gather*}
    (x-y)_{\kappa} := \left(1\wedge \frac{\kappa}{|x-y|}\right)(x-y),\quad \big(\tfrac 1\infty:=0\big).
\end{gather*}
The following \emph{refined basic coupling} was introduced in \cite[Section 2]{Lwcw} for the first time:
\begin{equation}\label{bas-1}
    (x,y)\longmapsto
    \begin{cases}
        (x+z, y+z+(x-y)_\kappa), & \frac12 \mu_{(y-x)_\kappa}(dz);\\
        (x+z, y+z+(y-x)_\kappa), & \frac12 \mu_{(x-y)_\kappa}(dz);\\
        (x+z, y+z), & \big(\nu - \frac12 \mu_{(y-x)_\kappa}  -\frac12 \mu_{(x-y)_\kappa}\big)(dz).
    \end{cases}
\end{equation}
Obviously, \eqref{bas-1} is the same as \eqref{basic-coup-4} if $n=2$, $\Psi_1(z)=z+(x-y)_\kappa$, $\Psi_2(z)=z+(y-x)_\kappa$ and $\nu_1=\nu_2=\frac{1}{2}\nu$. Since $\Psi_1^{-1}(z)=\Psi_2(z)$, \eqref{e:cou} holds true, and so \eqref{bas-1} yields a coupling operator.

Let us briefly discuss some properties of the  refined basic coupling \eqref{bas-1}.

If $|x-y|\leq \kappa$, then the refined basic coupling becomes
    \begin{equation}\label{bas-2}
        (x,y)\longmapsto
        \begin{cases}
        (x+z, y+z+(x-y)), & \frac12 \mu_{y-x}(dz);\\
        (x+z, y+z+(y-x)), & \frac12 \mu_{x-y}(dz);\\
        (x+z, y+z), & \big(\nu - \frac12 \mu_{y-x}  -\frac12 \mu_{x-y}\big)(dz).
        \end{cases}
    \end{equation}
    In the first row of \eqref{bas-2}, the distance of the two marginals decreases from $|x-y|$ to $|(x+z)- (y+z+(x-y))|=0$, and this reflects the idea of the basic coupling --   but only with half of the maximum common jump intensity from $x$ to $x+z$ and $y$ to $y+z+(x-y)$.   In the second row of \eqref{bas-2} the distance is doubled after jumping, with the remaining half of the maximum common jump intensity,   while we couple the remaining mass synchronously as indicated in the third row of \eqref{bas-2}.

If $|x-y|>\kappa$,
    then the first row of \eqref{bas-1} shows that the distance after the jump is $|x-y|-\kappa$. Therefore, the parameter $\kappa$ is the threshold to determine whether the marginal processes jump to the same point, or become slightly closer to each other. This is a technical point, but is crucial for our argument to make the coupling \eqref{bas-1} efficient for L\'evy jump processes with bounded  (finite-range) jumps.

Using the technique from \cite[Section 2.3]{Lwcw} we can construct the coupling process associated with the refined basic coupling. In a first step, we extend the Poisson random measure $N$ from  $\real_+\times \rd$ to $\real_+\times \rd\times [0,1]$  by adjoining an independent uniformly distributed random component
\begin{align*}
    N(ds,dz,du) &= \sum_{0<{r}\leq s : \Delta Z_{r}\neq 0}\delta_{(r, \Delta Z_{r})}(ds,dz)\,\I_{[0,1]}(u)\,du,\\
    \widetilde N(ds,dz,du) &= N(ds,dz,du)- ds\, \nu(dz)\,du,
\intertext{and we set}
    \breve{N}(ds,dz,du) & =\I_{[1,\infty)\times[0,1]}(|z|,u)\,N(ds,dz,du)+\I_{(0,1)\times [0,1]}(|z|,u)\,\widetilde{N}(ds,dz,du),\\
    Z_t &= \int_0^t\int_{\rd\times [0,1]} z\,\breve{N}(ds,dz,du).
\end{align*}
We are going to use the adjoined random variable to define a random threshold which determines whether the processes $X$ and $Y$ move towards each other or extend their distance. For this we need the following control function $\rho$:
\begin{gather*}
    \rho(x,z)
    = \frac{\nu \wedge (\delta_x\ast \nu) (dz)}{\nu(dz)}\in [0,1],\quad x,z\in\rd.
\end{gather*}
Recall that $(x)_\kappa=(1\wedge (\kappa |x|^{-1}))x$ for any $x\neq0$. Set $U_t=X_t-Y_t$,
\begin{gather*}
    V_{t}(z,u)
    = (U_{t})_{\kappa} \left( \I_{\left[0,  \frac 12\rho((-U_{t})_{\kappa},z)\right]}(u)
    - \I_{\left(\frac 12\rho((-U_{t})_{\kappa},z),\,\frac 12 [\rho((-U_{t})_{\kappa},z)+\rho((U_{t})_{\kappa},z)]\right]}(u)\right)
\end{gather*}
and
\begin{gather*}
    dL^{\#}_t = \int_{\rd\times [0,1]}V_{t-}(z,u) \,N(dt,dz,du).
\end{gather*}
Consider for any $x,y\in \rd$ with $x\neq y$ the following system of SDEs:
\begin{equation}\label{SDE-coup-eq-1}
    \left\{\begin{aligned}
        dX_t &= b(X_t)\, dt + dZ_t,&& X_0=x;\\
        dY_t &= b(Y_t)\,dt+ dZ_t + dL^{\#}_t, && Y_0=y.
  \end{aligned}\right.
\end{equation}
It is shown in \cite[Propositions 2.2 and 2.3]{Lwcw} that the system \eqref{SDE-coup-eq-1} has a unique strong solution which is a non-explosive coupling process $(X_t,Y_t)_{t\geq 0}$ of the SDE \eqref{s1}. Moreover, the generator of $(X_t,Y_t)_{t\geq 0}$ is the refined basic coupling operator constructed above, and $X_t=Y_t$ for any $t\geq \tau$, where $\tau=\inf\{t\geq 0: X_t=Y_t\}$ is the coupling time of the process $(X_t,Y_t)_{t\geq 0}$. Note that for $x\neq0$,
\begin{equation}\label{e:finitemeasure}
    \mu_x(\rd)
    \leq \int_{\{|z|\le|x|/2\}}\,\delta_x*\nu(dz)+\int_{\{|z|>|x|/2\}}\,\nu(dz)
    \leq 2 \int_{\{|z|\ge|x|/2\}}\,\nu(dz)<\infty,
\end{equation}
i.e.~$\mu_x$ is a finite measure on $(\rd, \Bscr(\rd))$ for any $x\neq0$.

\subsection{Coupling vs.\ optimal transport: rotationally symmetric L\'{e}vy noise}\label{section2-3}
In this section we discuss coupling from the point of view of optimal transport, see \cite{Vi} as a standard reference; our exposition is inspired by \cite[Section 2.1]{M15} and also by  McCann's solution to the optimal transport problem for concave costs on $\real$, cf.~\cite{Mc}. In contrast to \cite{M15} we will make full use of our approach through the coupling operator.

Let us return to the framework of Section~\ref{subse1}. We will assume that the pure jump L\'evy process $Z$ in the SDE \eqref{s1} is rotationally symmetric and that its L\'evy measure is of the form $\nu(dz):=q(|z|)\,dz$ for some nonnegative measurable function $q(r)$. Let $q_0(r)\leq q(r)$, i.e.\ $q_0(|z|)\,dz$ is also a rotationally symmetric L\'evy measure.

For any $x,y,z\in \rd$, let $R_{x,y}(z)$ be the reflection defined in \eqref{e:effrr}.  We consider the following jump system on $\real^{2d}$:
\begin{equation}\label{basic-coup-1}
(x,y)\longmapsto
    \begin{cases}
    (x+z, y+z+(x-y)), & q_0(|z|)\wedge q_0(|x-y+z|)\,dz;\\
    (x+z,y+R_{x,y}(z)), & \left[q_0(|z|)-q_0(|z|)\wedge q_0(|x-y+z|)\right]dz;\\
    (x+z,y+z),&\left[q(|z|)-q_0(|z|)\right]dz.
    \end{cases}
\end{equation}
If we choose in \eqref{basic-coup-4} $n=2$, $\Psi_1(z)=z+(x-y)$, $\Psi_2(z)=R_{x,y}(z)$,
$\nu_1(dz)=q_0(|z|)\wedge q_0(|x-y+z|)\,dz$ and
$\nu_2(dz)=\left[q_0(|z|)-q_0(|z|)\wedge q_0(|x-y+z|)\right]dz$, then \eqref{basic-coup-1} can be derived from \eqref{basic-coup-4}.
Observing that $R_{x,y}(x-y)=y-x$ ($x\neq y$) and $R_{x,y}(z_1+z_2)=R_{x,y}(z_1)+R_{x,y}(z_2)$ for any $z_1,z_2\in \rd$ we see
\begin{align*}
    &q_0(|\Psi_1^{-1}(z)|) \wedge q_0(|x-y+\Psi_1^{-1}(z)|)
    = q_0(|y-x+z|)\wedge q_0(|z|)
\intertext{and}
    &q_0(|R^{-1}_{x,y}(z)|)-q_0(|R^{-1}_{x,y}(z)|)\wedge q_0(|x-y+R^{-1}_{x,y}(z)|)\\
    &=q_0(|z|)-q_0(|z|)\wedge q_0(|R^{-1}_{x,y}(y-x+z)|)
    =q_0(|z|)-q_0(|z|)\wedge q_0(|y-x+z|).
\end{align*}
This shows that \eqref{e:cou} is satisfied.

We are now going to construct the coupling process for the coupling \eqref{basic-coup-1}. We continue to use the notation introduced in Section~\ref{subsection2}. Denote by $\bar N(dt,dz,du)$ and $\breve{N}_0(dt,dz,du)$ the extended Poisson random measure on $\real_+\times \rd\times [0,1]$ whose compensators are given by  $dt\,q(|z|)\,dz\,du$ and $dt\,q_0(|z|)\,dz\,du$, respectively. In order to keep notation simple, we set $r(z,X_{t-}-Y_{t-}) := (q_0(|z|)\wedge q_0(X_{t-}-Y_{t-}+z))/q_0(|z|)$.   Consider the following SDE:
\begin{equation}\label{e:sderef1}\left\{\begin{aligned}
    dX_t &= b(X_t)\,dt+\int_{\rd\times[0,1]} z\,\breve{N}(dt,dz,du),\\
    dY_t &= b(Y_t)\,dt+\int_{\rd\times[0,1]} z\,(\breve{N}(dt,dz,du)-\breve{N}_0(dt,dz,du))\\
    &\quad\mbox{}+\int_{\rd\times[0,1]}R_{X_{t-},Y_{t-}}(z)\I_{(r(z,X_{t-}-Y_{t-}),1]}(u)\,\breve{N}_0(dt,dz,du)\\
    &\quad\mbox{}+\int_{\rd\times[0,1]} (X_{t-}-Y_{t-}+z) \I_{(r(z,X_{t-}-Y_{t-}),1]}(u)\,\breve{N}_0(dt,dz,du)\\
    &\quad\mbox{}-\int_{\rd\times[0,1]}(X_{t-}-Y_{t-}+z) \left[\I_{(0,1)}(|z+(X_{t-}-Y_{t-})|)- \I_{(0,1)}(|z|)\right]\\
    &\qquad\mbox{}\times\I_{(r(z,X_{t-}-Y_{t-}),1]}(u)\,dt\,q_0(|z|)\,dz\,du.
\end{aligned}\right.\end{equation}

\begin{proposition}\label{P:strong}
    If $b$ is Lipschitz continuous on $\rd$, then the SDE \eqref{e:sderef1} has a unique strong solution up to the coupling time $\tau$. The generator of this solution is determined by the coupling \eqref{basic-coup-1}.
\end{proposition}

\begin{allowdisplaybreaks}
\begin{proof}
As in \eqref{e:finitemeasure}, we see that $\int q_0(|z|)\wedge q_0(|z+x-y|)\,dz<\infty$ for $x\neq y$. We can rewrite the SDE for $Y_t$ in the following way
\begin{align*}
    dY_t
    &= b(Y_t)\,dt + \int_{\rd\times[0,1]} z\,(\breve{N}(dt,dz,du)-\breve{N}_0(dt,dz,du))\\
    &\quad\mbox{} + \int_{\rd\times[0,1]} R_{X_{t-},Y_{t-}}(z)\,\breve{N}_0(dt,dz,du)\\
    &\quad\mbox{}-\int_{\rd\times[0,1]} R_{X_{t-},Y_{t-}}(z)\I_{(r(z,X_{t-}-Y_{t-}),1]}(u)\,\breve{N}_0(dt,dz,du)\\
    &\quad\mbox{} +\int_{\rd\times[0,1]} (X_{t-}-Y_{t-}+z)\I_{(r(z,X_{t-}-Y_{t-}),1]}(u)\,\breve{N}_0(dt,dz,du)\\
    &\quad\mbox{} -\int_{\rd\times[0,1]}(X_{t-}-Y_{t-}+z)\left[\I_{(0,1)}(|z+(X_{t-}-Y_{t-})|) - \I_{(0,1)}(|z|)\right]\\
    &\qquad\mbox{}\times\I_{(r(z,X_{t-}-Y_{t-}),1]}(u)\,dt\,q_0(|z|)\,dz\,du\\
    &= b(Y_t)\,dt + \int_{\rd\times[0,1]} z\,(\breve{N}(dt,dz,du)-\breve{N}_0(dt,dz,du))\\
    &\quad\mbox{} + \int_{\rd\times[0,1]} R_{X_{t-},Y_{t-}}(z)\,\breve{N}_0(dt,dz,du)\\
    &\quad\mbox{} - \int_{\rd\times[0,1]} R_{X_{t-},Y_{t-}}(z)\I_{(r(z,X_{t-}-Y_{t-}),1]}(u)\,{N}_0(dt,dz,du)\\
    &\quad\mbox{} + \int_{\rd\times[0,1]} (X_{t-}-Y_{t-}+z)\I_{(r(z,X_{t-}-Y_{t-}),1]}(u)\,{N}_0(dt,dz,du)\\
    &\quad\mbox{} + \int_{\{|z|\leq 1\}\times[0,1]} R_{X_{t-},Y_{t-}}(z)\I_{(r(z,X_{t-}-Y_{t-}),1]}(u)\,dt\,q_0(|z|)\,dz\,du\\
    &\quad\mbox{} - \int_{\{|z|\leq 1\}\times[0,1]} (X_{t-}-Y_{t-}+z)\I_{(r(z,X_{t-}-Y_{t-}),1]}(u)\,dt\,q_0(|z|)\,dz\,du\\
    &\quad\mbox{} - \int_{\rd\times[0,1]}(X_{t-}-Y_{t-}+z) \left[\I_{(0,1)}(|z+(X_{t-}-Y_{t-})|) - \I_{(0,1)}(|z|)\right]\\
    &\qquad\mbox{} \times \I_{(r(z,X_{t-}-Y_{t-}),1]}(u)\,dt\,q_0(|z|)\,dz\,du.
\end{align*}

We will now rearrange the last three terms involving  $dt\,q_0(|z|)\,dz\,du$:
\begin{align*}
    \int_{\{|z|< 1\}\times[0,1]} &R_{X_{t-},Y_{t-}}(z) \I_{(r(z,X_{t-}-Y_{t-}),1]}(u)\,dt\,q_0(|z|)\,dz\,du\\
    &= \int_{\{|z|< 1\}\times[0,1]}z\I_{[0,r(z,Y_{t-}-X_{t-})]}(u) \,dt\, q_0(|z|)\,dz\,du,
\end{align*}
which follows from $R_{x,y}(x-y)=y-x$ ($x\neq y$) and $R_{x,y}(z_1+z_2)=R_{x,y}(z_1)+R_{x,y}(z_2)$ for $z_1,z_2\in \rd$.
On the other hand,\begin{align*}
    &\int_{\rd\times[0,1]}(X_{t-}-Y_{t-}+z)\left[\I_{(0,1)}(|z+(X_{t-}-Y_{t-})|)- \I_{(0,1)}(|z|)\right]\\
    &\qquad\mbox{}\times\I_{(r(z,X_{t-}-Y_{t-}),1]}(u)\,dt\,q_0(|z|)\,dz\,du\\
    &=\int_{\{|z+(X_{t-}-Y_{t-})| < 1\}\times[0,1]}(X_{t-}-Y_{t-}+z)\I_{(r(z,X_{t-}-Y_{t-}),1]}(u)\,dt\,q_0(|z|)\,dz\,du\\
    &\quad\mbox{} - \int_{\{|z| < 1\}\times[0,1]} (X_{t-}-Y_{t-}+z)  \I_{(r(z,X_{t-}-Y_{t-}),1]}(u)\,dt\,q_0(|z|)\,dz\,du\\
    &=\int_{\{|z|< 1\}\times[0,1]}z \I_{[0,r(z,Y_{t-}-X_{t-})]}(u)\,dt\,q_0(|z|)\,dz\,du\\
    &\quad\mbox{} - \int_{\{|z|< 1\}\times[0,1]} (X_{t-}-Y_{t-}+z) \I_{(r(z,X_{t-}-Y_{t-}),1]}(u)\,dt\,q_0(|z|)\,dz\,du.
\end{align*}
 This means that the equation for $Y_t$ becomes simpler:
\begin{align*}
    dY_t
    &= b(Y_t)\,dt + \int_{\rd\times[0,1]} z\,(\breve{N}(dt,dz,du)-\breve{N}_0(dt,dz,du))\\
    &\quad\mbox{} +\int_{\rd\times[0,1]} R_{X_{t-},Y_{t-}}(z)\,\breve{N}_0(dt,dz,du)\\
    &\quad\mbox{} -\int_{\rd\times[0,1]} R_{X_{t-},Y_{t-}}(z)\I_{(r(z,X_{t-}-Y_{t-}),1]}(u)\,{N}_0(dt,dz,du)\\
    &\quad\mbox{} +\int_{\rd\times[0,1]} (X_{t-}-Y_{t-}+z)\I_{(r(z,X_{t-}-Y_{t-}),1]}(u)\,{N}_0(dt,dz,du).
\end{align*}

For fixed $z\in \rd$, the function $(x,y)\mapsto R_{x,y}(z)$ is locally Lipschitz continuous on $\{(x,y)\in \real^{2d}:x\neq y\}$. The remaining two terms driven by
\begin{gather*}
    \I_{(r(z,X_{t-}-Y_{t-}),1]}(u)\,{N}_0(dt,dz,du)
\end{gather*}
may be regarded as stochastic integrals with respect to a finite Poisson measure; this is again due to the fact that $\int q_0(|z|)\wedge q_0(z+x-y)\,dz<\infty$ for any $x\neq y$. Using the standard interlacing technique, we see that the SDE \eqref{e:sderef1} has a unique strong solution up to the coupling time $\tau$, see \cite[Chapter~IV.9]{ike-wat}.
\end{proof}
\end{allowdisplaybreaks}

\section{Regularity of SDEs with additive L\'evy noise revisited}\label{section3}
We will now apply our coupling technique, to study regularity properties of the semigroup associated with the SDE \eqref{s1}. Let $(X_t)_{t\geq 0}$ be the unique strong solution to \eqref{s1} and denote by $(P_t)_{t\geq 0}$ its transition semigroup. We define
\begin{gather*}
    B(r) := \frac 1r \sup_{|x-y|=r} \scalp{b(x)-b(y)}{x-y}, \quad r>0.
\end{gather*}

\begin{theorem}\label{T1}
Let $Z$ be a pure jump L\'evy process with L\'evy measure $\nu$.
\begin{enumerate}[label=\textup{\alph*)}]
\item\label{T1-a}
    Let $Z$ be rotationally symmetric. Define
    \begin{gather*}
        \psi(r) = \int_{\{|z|\leq r\}}|z|^2\,\nu(dz)
        \et
        \Phi(r) = \int_0^{r}\int_u^1\frac{1}{\psi(s/4)}\,ds\,du,
        \quad r>0.
    \end{gather*}
    If
    \begin{equation}\label{e:tss}
        \textup{a)}\;\;\int_0^1\frac{s}{\psi(s)}\,ds<\infty
        \et
        \textup{b)}\;\;\limsup_{r\to 0}\left( B(r)\int_r^1\frac{1}{\psi(s/4)}\,ds\right)< \frac 2d,
    \end{equation}
    then there exists a constant $c>0$ such that for any $f\in B_b(\rd)$, $x\in \rd$ and $t>0$,
    \begin{gather*}
        \limsup_{y\to x} \frac{|P_tf(x)-P_tf(y)|}{\Phi(|x-y|)}\leq c\left(1\wedge \frac{1}{t} \right).
    \end{gather*}

\item\label{T1-b}
    For an arbitrary pure jump L\'evy process $Z$, define
    \begin{gather}\label{e:tss2}
        \psi(r)=r^2\inf_{|x|\leq r}(\nu \wedge (\delta_x*\nu))(\rd)
        \et
        \Phi(r)=\int_0^{r}\int_u^1\frac{1}{\psi(s/2)}\,ds\,du,
        \quad r>0.
    \end{gather}
    If
    \begin{equation}
        \textup{a)}\;\;\int_0^1\frac{s}{\psi(s)}\,ds<\infty
        \et
        \textup{b)}\;\;\limsup_{r\to 0}\left( B(r)\int_r^1\frac{1}{\psi(s/2)}\,ds\right)< \frac 12,
    \end{equation}
    then there exists a constant $c>0$ such that for any $f\in B_b(\rd)$, $x\in \rd$ and $t>0$,
    \begin{gather*}
        \limsup_{y\to x} \frac{|P_tf(x)-P_tf(y)|}{\Phi(|x-y|)}\leq c\left(1\wedge \frac{1}{t} \right).
    \end{gather*}
\end{enumerate}
\end{theorem}

\begin{remark}
a) Recently, the authors of \cite{KS1} introduced the \emph{local coupling property} for Markov processes and proved that it is equivalent to the following condition
\begin{equation}\label{e:t1}
    \lim_{y\to x}\|p_t(x,\cdot)-p_t(y,\cdot)\|_{\var}=0
    \quad\text{for all $x\in \rd$ and $t>0$},
\end{equation}
where $p_t(x,\cdot)$ is the transition probability of the Markov process $X:=(X_t)_{t\geq 0}$. Let $(P_t)_{t\geq 0}$ be the transition semigroup associated of the process $X$. Since
\begin{gather*}
    \|p_t(x,\cdot)-p_t(y,\cdot)\|_{\var}
    = \sup_{\|f\|_\infty\leq 1}|P_tf(x)-P_tf(y)|,
\end{gather*}
\eqref{e:t1} implies that $(P_t)_{t\geq 0}$ is a strong Feller semigroup, i.e.\ for any $t>0$, $P_t$ maps the set of bounded measurable functions into the set of bounded continuous functions. On the other hand, using \cite[Chapter 1, Propositions 5.8 and 5.12]{Rev}, we can deduce \eqref{e:t1} from the strong Feller property. Therefore, the local coupling property and the strong Feller property coincide for Markov semigroups.

\medskip\noindent
b) Since we have $\Phi(0)=0$, $\Phi'(r)>0$ and $\Phi''(r)<0$ on $(0,\infty)$, $(x,y)\mapsto \Phi(|x-y|)$ is a distance function in $\rd$. Therefore, Theorem~\ref{T1} guarantees the regularity of the semigroup associated with the SDE \eqref{s1} which in turn implies the strong Feller property.

\medskip\noindent
c) If $b\equiv 0$, i.e.\ if $X=Z$ is a pure jump L\'evy process, the conditions (\ref{e:tss}.b) and (\ref{e:tss2}.b) are trivially satisfied. Theorem~\ref{T1}.\ref{T1-b} seems to be new even for L\'evy processes. If  $X=Z$  is a rotationally symmetric L\'evy process, Theorem~\ref{T1}.\ref{T1-a} also extends \cite[Theorem 2.2]{KS}, where only the one-dimensional case is discussed.

\medskip\noindent
d) According to \cite[Example 1.2]{Lwcw}, Theorem~\ref{T1}.\ref{T1-b} holds for any L\'evy measure $\nu$ satisfying
\begin{gather*}
    \nu(dz)
    \geq \I_{(0,1]}(z_1)\frac{c}{|z|^{d+\alpha}}\,dz
\end{gather*}
for some $\alpha\in (0,2)$ and $c>0$. If we take, for example, $\nu(dz)=\I_{(0,1]}(z_1)\frac{c}{|z|^{d+\alpha}}\,dz$, then $\nu$ is a L\'evy measure with zero symmetric part, and such settings are not covered by \cite[Theorem 7]{KS1} which treats only the one-dimensional case.
\end{remark}

In the proof of Theorem~\ref{T1}.\ref{T1-a} we will apply the coupling operator $\widetilde L$ from Section~\ref{subse1} with $\eta= \frac 12$. We begin with the following simple estimate.
\begin{lemma}\label{L1}
    Let $\widetilde L$ be the coupling operator given by the jump system \eqref{basic-coup-5} with $\eta= \frac 12$. Pick $f\in C[0,2]\cap C^2(0,2]$ such that $f(0)=0$, $f'\geq 0$, $f''\leq 0$ and $f''$ is increasing on $(0,2]$. For any $x,y\in \rd$ with $0<|x-y|\leq 1$,
    \begin{gather*}
        \widetilde L f(|x-y|)
        \leq f'(|x-y|) \frac{\scalp{b(x)-b(y)}{x-y}}{|x-y|} + \frac{2}{d}f''(2|x-y|) \int_{\{|z|\leq |x-y|/2\}}|z|^2\,\nu(dz).
    \end{gather*}
\end{lemma}
\begin{proof}
Let $f\in C[0,2]\cap C^2(0,2]$ with $f(0)=0$ and $x,y\in \rd$ with $0<|x-y|\leq 1$.
\begin{align*}
    \widetilde L f(|x-y|)
    &= f'(|x-y|) \frac{\scalp{b(x)-b(y)}{x-y}}{|x-y|}\\
    &\qquad\mbox{}+\int_{\{|z|\leq |x-y|/2\}}\bigg[f(|(x-y)+(z-R_{x,y}(z))|)-f(|x-y|)\\
    &\qquad\qquad\qquad\mbox{}-f'(|x-y|)\frac{\scalp{x-y}{z}}{|x-y|}\I_{(0,1)}(|z|)\\
    &\qquad\qquad\qquad\qquad\qquad\mbox{}+f'(|x-y|)\frac{\scalp{x-y}{R_{x,y}(z)}}{|x-y|}\I_{(0,1)}(|z|)\bigg]\,\nu(dz),
\end{align*}
where we use $|R_{x,y}(z)|=|z|$. Observe that $\widetilde Lf(|x-y|) = \widetilde Lf(|y-x|)$ and $R_{xy}(z)=R_{yx}(z)$. This allows us to symmetrize the above expression and we get
\begin{align*}
    \widetilde L f(|x-y|)
    &= f'(|x-y|) \frac{\scalp{b(x)-b(y)}{x-y}}{|x-y|}\\
    &\qquad\qquad\mbox{}+\frac{1}{2}\int_{\{|z|\leq |x-y|/2\}}\bigg[f(|(x-y)+(R_{x,y}(z) -z)|)\\
    &\qquad\qquad\qquad\mbox{}+f(|(x-y)+(z-R_{x,y}(z))|)-2f(|x-y|) \bigg]\,\nu(dz)\\
    &=f'(|x-y|) \frac{\scalp{b(x)-b(y)}{x-y}}{|x-y|}\\
    &\qquad\qquad\mbox{}+\frac{1}{2}\int_{\{|z|\leq |x-y|/2\}}\bigg[f\left(|x-y|\left(1+\frac{2\scalp{x-y}{z}}{|x-y|^2}\right)\right)\\
    &\qquad\qquad\qquad\mbox{}+f\left(|x-y|\left(1-\frac{2\scalp{x-y}{z}}{|x-y|^2}\right)\right)-2f(|x-y|)\bigg]\,\nu(dz);
\end{align*}
to see the last equality, use that $|z|\leq |x-y|/2$.

We assume now, in addition, that $f'\geq 0$, $f''\leq 0$ and $f''$ is increasing. For any $\delta\in [0,r]$,
\begin{equation}\label{e:est1}
    f(r+\delta)+f(r-\delta)-2f(r)
    = \int_r^{r+\delta}\int_{s-\delta}^s f''(u)\,du\,ds
    \leq f''(r+\delta)\delta^2.
\end{equation}
Using again the fact that $\nu$ is rotationally symmetric, we get
\begin{align*}
    &\widetilde L f(|x-y|)\\
    &\leq f'(|x-y|) \frac{\scalp{b(x)-b(y)}{x-y}}{|x-y|}+ 2f''(2|x-y|)\int_{\{|z|\leq |x-y|/2\}}\frac{\scalp{x-y}{z}^2}{|x-y|^2}\,\nu(dz)\\
    &= f'(|x-y|) \frac{\scalp{b(x)-b(y)}{x-y}}{|x-y|} + 2f''(2|x-y|)\int_{\{|z|\leq |x-y|/2\}}|z_1|^2\,\nu(dz)\\
    &= f'(|x-y|) \frac{\scalp{b(x)-b(y)}{x-y}}{|x-y|} + \frac{2}{d}f''(2|x-y|)\int_{\{|z|\leq |x-y|/2\}}|z|^2\,\nu(dz).
\qedhere
\end{align*}
\end{proof}

\begin{proof} [Proof of Theorem~\ref{T1}.\ref{T1-a}]
Fubini's theorem shows
\begin{gather*}
    \Phi(r)
    = \int_0^{r}\int_u^1\frac{ds\,du}{\psi(s/4)}
    = \int_0^1\frac{s\wedge r}{\psi(s/4)}\,ds,\quad r>0.
\end{gather*}
Since  $r\mapsto \psi(r)$ is increasing, $\Phi$ is well defined under (\ref{e:tss}.a); moreover, $\Phi'(r)=\int_r^1\frac{ds}{\psi(s/4)}> 0$,  $\Phi''(r)=- \frac 1{\psi(r/4)}<0$, and $\Phi''$ is increasing since $r\mapsto \psi(r)$ is increasing. According to Lemma~\ref{L1} and (\ref{e:tss}.b), there exist constants $\epsilon_0\in (0,1]$ and $c_0>0$ such that for $x,y\in \rd$ with $0<|x-y|\leq \epsilon_0$,
\begin{equation}\label{e:sregu}
    \widetilde L \Phi(|x-y|)\leq -c_0.
\end{equation}
Let $(X_t,Y_t)_{t\geq 0}$ be the coupling process constructed at the end of Section~\ref{subse1}. Denote by $\widetilde\Pp^{(x,y)}$ and $\widetilde\Ee^{(x,y)}$ the probability law and the expectation of $(X_t,Y_t)_{t\geq 0}$ such that $(X_0,Y_0)=(x,y)$, respectively. For $\epsilon_0\in (0,1]$ as above and any $n\geq 1$ we set
\begin{align*}
    \sigma_{\epsilon_0}
    &:=\inf\left\{t\geq 0: |X_t-Y_t|>\epsilon_0\right\},\\
    \tau_n
    &:=\inf\left\{t\geq 0: |X_t-Y_t|\leq 1/n\right\}.
\end{align*}
It is clear that $\lim_{n\to \infty}\tau_n=\tau$, where $\tau$ is the coupling time.

Let
$x,y\in\rd$ with $0<|x-y|<\epsilon_0$ and choose $n$ so large that $|x-y|>1/n$. Because of the monotonicity of $\Phi$, Dynkin's formula and \eqref{e:sregu}, we get for all $t>0$,
\begin{align*}
  0
  &\leq \Phi(\epsilon_0)\Pp^{(x,y)}(\sigma_{\epsilon_0}<\tau_n\wedge t)\\
  &\leq\widetilde\Ee^{(x,y)}\Phi\left(|X_{t\wedge \tau_n\wedge \sigma_{\epsilon_0}}-Y_{t\wedge \tau_n\wedge \sigma_{\epsilon_0}}|\right)\\
  &=\Phi(|x-y|)+\widetilde\Ee^{(x,y)}\left(\int_0^{t\wedge \tau_n\wedge \sigma_{\epsilon_0}} \widetilde{L} \Phi\big(|X_{s}-Y_{s}|\big)\,ds\right)\\
  &\leq \Phi(|x-y|)-c_0\widetilde\Ee^{(x,y)}(t\wedge\tau_n\wedge \sigma_{\epsilon_0}).
\end{align*}
Rearranging this inequality and letting $t\to\infty$ yields
\begin{align*}
    c_0\widetilde\Ee^{(x,y)}(\tau_n\wedge \sigma_{\epsilon_0}) + \Phi(\epsilon_0)\Pp^{(x,y)}(\sigma_{\epsilon_0}<\tau_n)
    \leq \Phi(|x-y|).
\end{align*}
Therefore, we can use Markov's inequality and get
\begin{align*}
    \widetilde\Pp^{(x,y)}(\tau_n> t)
    &\leq \widetilde\Pp^{{(x,y)}}(\tau_n\wedge \sigma_{\epsilon_0}>t) + \widetilde\Pp^{{(x,y)}}(\sigma_{\epsilon_0}<\tau_n)\\
    &\leq \frac{\widetilde\Ee^{(x,y)}(\tau_n\wedge \sigma_{\epsilon_0})}{t}+\frac{\Phi(|x-y|)}{\Phi({\epsilon_0})}
    \leq \Phi(|x-y|)\left[\frac{1}{tc_0} + \frac{1}{\Phi({\epsilon_0})}\right].
\end{align*}
Letting $n\to \infty$, we find that
\begin{gather*}
    \widetilde\Pp^{(x,y)}(\tau> t)\leq \Phi(|x-y|)\left[\frac{1}{tc_0} + \frac{1}{\Phi({\epsilon_0})}\right].
\end{gather*}
Finally, we have for any $f\in B_b(\rd)$, $x\in \rd$ and $t>0$,
\begin{align*}
    \limsup_{y\to x} \frac{|P_tf(x)-P_tf(y)|}{\Phi(|x-y|)}
    &= \limsup_{y\to x} \frac{|\widetilde\Ee^{{(x,y)}}[f(X_t)-f(Y_t)]|}{\Phi(|x-y|)}\\
    &= \limsup_{y\to x} \frac{|\widetilde\Ee^{{(x,y)}}[(f(X_t)-f(Y_t))\I_{\{\tau>t\}}]| }{\Phi(|x-y|)}\\
    &\leq 2 \|f\|_\infty \limsup_{y\to x} \frac{\widetilde\Pp^{{(x,y)}}(\tau>t) }{\Phi(|x-y|)}\\
    &\leq 2 \left[\frac{1}{tc_0}+\frac{1}{\Phi({\epsilon_0})}\right]
    \leq c_1\left(1\wedge\frac{1}{t}\right).
\qedhere
\end{align*}
 \end{proof}

In order to prove Theorem~\ref{T1}.\ref{T1-b}, we use the refined basic coupling from Section~\ref{subsection2}.
\begin{lemma}\label{L2}
    Let $\kappa>0$ be the constant from \eqref{bas-1} and denote by $\widetilde L$ the coupling operator given by the jump system \eqref{bas-1}. Let $f\in C[0,2\kappa]\cap C^2(0,2\kappa]$ such that $f(0)=0$, $f'\geq 0$, $f''\leq 0$ and $f''$ is increasing on $(0,2\kappa]$. For all $x,y\in \rd$ with $0<|x-y|\leq \kappa$,
    \begin{gather*}
        \widetilde L f(|x-y|)
        \leq f'(|x-y|) \frac{\scalp{b(x)-b(y)}{x-y}}{|x-y|} + \frac{1}{2}\mu_{(x-y)}(\rd) |x-y|^2 f''(2|x-y|),
    \end{gather*}
    where $\mu_{x}(dz)=[\nu\wedge (\delta_{x} \ast\nu)](dz)$ for all $x\in \rd$.
\end{lemma}
\begin{proof}
If $x,y\in \rd$ satisfy $0<|x-y|\leq \kappa$, then we have $(x-y)_\kappa=(x-y)$. Using the jump system \eqref{bas-1}, we get for any $x,y\in \rd$ with $0<|x-y|\leq \kappa$
\begin{align*}
    \widetilde Lf(|x-y|)
    &= f'(|x-y|) \frac{\scalp{b(x)-b(y)}{x-y}}{|x-y|}\\
    &\quad\mbox{}+\left[\frac{1}{2} \mu_{(y-x)}(\rd)\left(f(2|x-y|)-f(|x-y|)\right) - \frac{1}{2}\mu_{(y-x)}(\rd) f(|x-y|)\right]\\
    &= f'(|x-y|) \frac{\scalp{b(x)-b(y)}{x-y}}{|x-y|}\\
    &\quad\mbox{} + \frac{1}{2}\mu_{(x-y)}(\rd)\left[f(2|x-y|)-2f(|x-y|)\right],
\end{align*}
where we use the identity $\mu_{(x-y)}(\rd)=\mu_{(y-x)}(\rd)$.
This, together with the assumptions on $f$ and \eqref{e:est1}, yields the assertion.
\end{proof}

Theorem~\ref{T1}.\ref{T1-b} can now be proved along the same lines as Theorem~\ref{T1}.\ref{T1-a}: all we have to do is to use Lemma~\ref{L2} instead of Lemma~\ref{L1}. Since the arguments are similar, we leave the details to the reader.

\section{Optimal coupling operators \& coupling operators for SDEs with multiplicative L\'evy noise}\label{section4}

\subsection{Optimal coupling operators for L\'evy processes}

As mentioned earlier, the construction of \eqref{basic-coup-1} is motivated by optimal coupling of Gaussian distributions, see \cite[Theorem 1.4]{GM} or \cite[Section 3]{HS}, and reflection coupling for Brownian motion, see \cite[Section 2]{LR} or \cite[Section 2]{HS}.  Folklore wisdom from the theory of optimal transport tells us that one should use most of the common mass of two probability distributions if one wants to obtain a coupling with nice properties.
In this sense, the first row in \eqref{basic-coup-1} is a natural choice, see \cite[Section 2.1]{M15} for further details,  and this is also the underlying idea of basic coupling \eqref{bas-2}.
The problem is, how one should use the remaining mass.

If the L\'evy measure is rotationally symmetric, we use reflection of the remaining mass, cf.~the second row of \eqref{basic-coup-1}. This approach is essentially due to \cite[Section 2.2]{M15}, where $q_0(|z|)=q(|z|)\I_{\{|z|< m\}}$ for some large $m\gg 1$. For L\'evy processes which are subordinate to a Brownian motion, \cite{B} shows that this type of coupling with $m=\infty$ is indeed a Markovian maximal coupling. For further discussions in this direction, we refer our readers to the end of \cite[Section 5]{B}.

In a general setting, one can try to use independent coupling with the remaining mass; this approach often has poor properties. Intuitively, a much better solution should be to couple the remaining mass synchronously, but it turns out that such a construction does not produce a coupling. In the preliminary construction of the refined basic coupling \eqref{bas-2}, the two marginal processes jump to the same place only with \emph{half} of the maximal probability (see the first row in \eqref{bas-2}), while with the other half we perform  a transformation which doubles the distance between the two marginal processes (see the second row in \eqref{bas-2}). With the remaining probability we let the marginal processes move synchronously, see the third row in \eqref{bas-2}. With a view towards the refined basic coupling \eqref{bas-2}, it seems sometimes to be better not to have the marginals jump to the same place with the \emph{maximal possible} probability, but to use some of the mass for coupling the marginals in a more convenient way.

In what follows, we use the concept of an \emph{optimal coupling operator} which was introduced in \cite[Definition 2.3]{Chen1} to study optimal coupling for L\'evy noise. Let $f$ be a non-decreasing and concave function $[0,\infty)$ such that $f(0)=0$, a coupling operator $\bar{L}_f$ is said to be \emph{$f$-optimal}, if
\begin{gather*}
    \bar{L}_f f(|x_1-x_2|)=\inf_{\widetilde L} \widetilde L f(|x_1-x_2|),\quad x_1,x_2\in \rd,
\end{gather*}
where the infimum ranges over all coupling operators $\widetilde L$. In particular, the definition above means that the infimum is attained for the coupling operator $\bar L_f$. In contrast to the diffusion or the birth-and-death process case -- see \cite[Theorem~3.2 and Section~5]{Chen1} -- there seems to be no general structure formula for coupling operators associated with L\'evy noise. This is the reason, why we concentrate on the three couplings presented in Section~\ref{section2}: we will compare $\widetilde Lf$ ($f$ is a non-decreasing and concave function on $[0,\infty)$ such that $f(0)=0$) with the three coupling operators mentioned in Section~\ref{section2}.

Let $Z=(Z_t)_{t\geq 0}$ be a rotationally symmetric  pure jump  L\'evy process whose L\'evy measure is of the form $\nu(dz)=q(|z|)\,dz$ for some measurable function $q(r)\ge0$, $r>0$.  We consider the following two cases.

\medskip\noindent
\textbf{Case 1:} (Jumps of infinite range)
Denote by $\widetilde L_{r}$ the ``coupling-by-reflection'' operator with $\eta=\infty$, cf.~Section~\ref{subse1}, and by $\widetilde L_{r,b}$ the ``combined reflection-and-basic'' coupling operator constructed in Section~\ref{section2-3} with $q_0(|z|)=q(|z|)$. For any $f\in C[0,\infty)\cap C^2(0,\infty)$ with $f(0)=0$ and $f'\geq 0$, and any $x,y\in \rd$ with $x\neq y$, we have
\begin{align*}
    &\widetilde L_{r} f(|x-y|)\\
    &=\int_\rd\bigg[f\left(|x-y+z-R_{x,y}(z)|\right)-f(|x-y|) -f'(|x-y|)\frac{\scalp{x-y}{z}}{|x-y|}\I_{(0,1)}(|z|)\\
    &\qquad\qquad+f'(|x-y|)\frac{\scalp{x-y}{R_{x,y}(z)}}{|x-y|}\I_{(0,1)}(|z|) \bigg]\,q(|z|)\,dz.
\end{align*}
Since $\widetilde L_{r} f(|x-y|)=\widetilde L_{r} f(|y-x|)$ and $R_{x,y}(z)=R_{y,x}(z)$, we can symmetrize the above expression, and get
\begin{align*}
    &\widetilde L_{r} f(|x-y|)\\
    &=\frac{1}{2}\int_\rd
    \left[f\left(|x-y+z-R_{x,y}(z)|\right)+f\left(|x-y+R_{x,y}(z)-z|\right) -2f(|x-y|)\right]q(|z|)\,dz.
\end{align*}
For the other coupling operator we get
\begin{align*}
    &\widetilde L_{r,b} f(|x-y|)\\
    &= \int_\rd\bigg[f\left(|x+z-y-z-(x-y)|\right) - f(|x-y|) - f'(|x-y|)\frac{\scalp{x-y}{z}}{|x-y|}\I_{(0,1)}(|z|)\\
    &\quad\mbox{}+f'(|x-y|)\frac{\scalp{x-y}{z+(x-y)}}{|x-y|}\I_{(0,1)}(|z+(x-y)|)\bigg] q(|z|)\wedge q(|z+x-y|)\,dz\\
    &\quad\mbox{}+\int_\rd \bigg[f\left(|x+z-y-R_{x,y}(z)|\right) - f(|x-y|) - f'(|x-y|)\frac{\scalp{x-y}{z}}{|x-y|}\I_{(0,1)}(|z|)\\
    &\quad\mbox{}+f'(|x-y|)\frac{\scalp{x-y}{R_{x,y}(z)}}{|x-y|}\I_{(0,1)}(|z|) \bigg] \big[q(|z|)-q(|z|)\wedge q(|z+x-y|)\big]\,dz\\
    &=-\int_\rd f\left(|x+z-y-R_{x,y}(z)|\right) q(|z|) \wedge q(|z+x-y|)\,dz\\
    &\quad\mbox{}+\int_\rd \bigg[f\left(|x+z-y-R_{x,y}(z)|\right) - f(|x-y|) - f'(|x-y|)\frac{\scalp{x-y}{z}}{|x-y|}\I_{(0,1)}(|z|)\\
    &\quad\mbox{}+f'(|x-y|)\frac{\scalp{x-y}{R_{x,y}(z)}}{|x-y|}\I_{(0,1)}(|z|) \bigg]\,q(|z|)\,dz.
\end{align*}
In the last equality we use that $q(z)\wedge q(z+x-y)\,dz$, $x\neq y$, is a finite measure on $\rd$ as well as the following identity which one easily checks using (in the last line) the change of variables $z\rightsquigarrow R_{x,y}(z)$ and $R_{x,y}(x-y)=R_{x,y}(y-x)$:
\begin{align*}
    &\int\frac{\scalp{x-y}{z+(x-y)}}{|x-y|}\I_{(0,1)}(|z+(x-y)|)\, q(|z|)\wedge q(|z+x-y|)\,dz\\
    &\quad=\int\frac{\scalp{x-y}{z}}{|x-y|}\I_{(0,1)}(|z|) \, q(z-x+y)\wedge q(|z|)\,dz \\
    &\quad= \int\frac{\scalp{x-y}{R_{x,y}(z)}}{|x-y|} \I_{(0,1)}(|z|)\, q(|z|)\wedge q(|z+x-y|)\,dz.
\end{align*}
Using symmetry as above, we can swap the roles of $x$ and $y$ in the second term of the right hand side for $\widetilde L_{r,b} f(|x-y|)$, and get
\begin{align*}
    &\widetilde L_{r,b} f(|x-y|)\\
    &=-\int_\rd f\left(|x+z-y-R_{x,y}(z)|\right) q(|z|) \wedge q(|z+x-y|)\,dz\\
    &\quad\mbox{}+\frac{1}{2}\int_\rd \big[f\left(|x -y+z-R_{x,y}(z)|\right) + f\left(|x-y+R_{x,y}(z)-z|\right) - 2f(|x-y|)\big]\, q(|z|)\,dz.
\end{align*}
Comparing the formulae for $\widetilde L_{r} f(|x-y|)$ and $\widetilde L_{r,b} f(|x-y|)$ we see that
\begin{gather*}
    \widetilde L_{r,b} f(|x-y|)
    \leq \widetilde L_{r} f(|x-y|),\quad x,y\in \rd.
\end{gather*}

\medskip\noindent
\textbf{Case 2:} (Jumps of finite range)
Denote by $\widetilde L_{r}$ the ``coupling-by-reflection'' operator with $\eta=\frac 12$, cf.~Section~\ref{subse1}, and by $\widetilde L_{b}$ the ``refined-basic-coupling'' operator constructed in the same way as \eqref{bas-1} with $\nu_1(dz)=\nu_2(dz)= \frac{1}{2}\I_{\{|z|\leq |x-y|/2\}}\,q(|z|)\,dz$ in Section~\ref{subsection2}. For  any $f\in C[0,\infty)\cap C^2(0,\infty)$ with $f(0)=0$ and $f''\leq 0$, we get with the symmetrization argument for $\widetilde L_{r}$ from Case~1 that
\begin{align*}
    \widetilde L_{r} f(|x-y|)
    = \frac{1}{2}\int_{\{|z|\leq |x-y|/2\}}\bigg[&f\left(|x-y|+\frac{2\scalp{x-y}{z}}{|x-y|}\right)\\
    &\mbox{}+f\left(|x-y|-\frac{2\scalp{x-y}{z}}{|x-y|}\right)-2f(|x-y|)\bigg]\,q(|z|)\,dz
\end{align*}
and
\begin{align*}
    \widetilde L_{b} f(|x-y|)
    &=\frac{1}{2}\big[f\left(|x-y|-|x-y|\wedge\kappa\right) +f\left(|x-y|+|x-y|\wedge\kappa\right)-2f(|x-y|)\big]\\
    &\qquad\mbox{}\times \int [q(|z|)\I_{\{|z|\leq |x-y|/2\}}] \wedge [q(|z+x-y|)\I_{\{|z+x-y|\leq |x-y|/2\}}]\,dz.
\end{align*}
By the mean value theorem and $f''\leq 0$, it is easy to see that
\begin{gather*}
    \widetilde L_{r} f(|x-y|)\leq 0.
\end{gather*}
On the other hand, we have for all $x,y,z\in \rd$
\begin{equation}\label{e:zero}
\begin{split}
    0
    &\leq [q(|z|)\I_{\{|z|\leq |x-y|/2\}}] \wedge [q(|z+x-y|)\I_{\{|z+x-y|\leq |x-y|/2\}}] \\
    &\leq [q(|z|)\I_{\{|z|\leq |x-y|/2\}}] \wedge [q(|z+x-y|)\I_{\{|z|\geq |x-y|/2\}}(z)] = 0
    \quad \text{a.e.}
\end{split}\end{equation}
This shows that
\begin{gather*}
    \widetilde L_{b} f(|x-y|)=0,\quad x,y\in\rd.
\end{gather*}

Let $\widetilde L_{r,b}$ denote the ``combined reflection-and-basic'' coupling operator constructed in Section~\ref{section2-3} with $q_0(|z|)=\I_{\{|z|\leq |x-y|/2\}}\,q(|z|)$.  Using \eqref{e:zero} it is easy to see that
\begin{align*}
    &\widetilde L_{r,b} f(|x-y|)\\
    &=\int \bigg[f\left(|x+z-y-z-(x-y)|\right) - f(|x-y|) - f'(|x-y|)\frac{\scalp{x-y}{z}}{|x-y|}\I_{(0,1)}(|z|)\\
    &\qquad\mbox{} + f'(|x-y|)\frac{\scalp{x-y}{z+x-y}}{|x-y|}\I_{(0,1)}(|z+(x-y)|)\bigg]\\
    &\qquad\quad\mbox{}\times \left[q(|z|)\I_{\{|z|\leq |x-y|/2\}}\right]\wedge \left[q(|z+x-y|)\I_{\{|z+x-y|\leq |x-y|/2\}}\right]dz\\
    &\quad\mbox{}+\int_{\{|z|\leq |x-y|/2\}}  \bigg[f\left(|x-y|+\frac{2\scalp{x-y}{z}}{|x-y|}\right) - f(|x-y|)\\
    &\qquad\mbox{}-f'(|x-y|)\frac{\scalp{x-y}{z}}{|x-y|}\I_{(0,1)}(|z|) + f'(|x-y|)\frac{\scalp{x-y}{R_{x,y}(z)}}{|x-y|}\I_{(0,1)}(|z|)\bigg]\\
    &\qquad\quad\mbox{}\times \left(q(|z|)-[q(|z|)\I_{\{|z|\leq |x-y|/2\}}]\wedge [q(|z+x-y|)\I_{\{|z+x-y|\leq |x-y|/2\}}]\right)dz\\
    &=\int_{\{|z|\leq |x-y|/2\}}  \bigg[f\left(|x-y| + \frac{2\scalp {x-y}{z}}{|x-y|}\right) - f(|x-y|)\\
    &\quad\mbox{}-f'(|x-y|)\frac{\scalp{x-y}{z}}{|x-y|}\I_{(0,1)}(|z|) + f'(|x-y|)\frac{\scalp{x-y}{R_{x,y}(z)}}{|x-y|}\I_{(0,1)}(|z|) \bigg]q(|z|)\,dz.
\end{align*}
The symmetrization argument used in Case 1 allows us to change the roles of $x$ and $y$, and we get
\begin{align*}
    \widetilde L_{r,b} f(|x-y|)
    &=\frac{1}{2}\int_{\{|z|\leq |x-y|/2\}} \bigg[f\left(|x-y|+\frac{2\scalp{x-y}{z}}{|x-y|}\right)\\
    &\qquad\qquad\qquad\qquad +f\left(|x-y|-\frac{2\scalp{x-y}{z}}{|x-y|}\right)-2f(|x-y|)\bigg]\,q(|z|)\,dz\\
    &=\widetilde L_r f(|x-y|).
\end{align*}

If $f\in C[0,\infty)\cap C^2(0,\infty)$ with $f(0)=0$ and $f''\leq 0$, these calculations show that
\begin{gather*}
    \widetilde L_{r,b} f(|x-y|)
    = \widetilde L_{r} f(|x-y|)
    \leq \widetilde L_{b} f(|x-y|)=0
    \quad\text{ for all $x,y\in \rd$}.
\end{gather*}

\begin{remark}
    L\'evy processes which are subordinate to a Brownian motions are particular examples of rotationally symmetric L\'evy processes. Thus, the conclusion of Case~1 shows that the coupling defined by \eqref{basic-coup-1} is, for subordinated Brownian motions and from a coupling operator point-of-view, \emph{optimal} among the three couplings mentioned in Section~\ref{section2}.

    On the other hand, one essential point of the proof in Case 2 uses the fact that, when L\'evy jump is finite range, the jumping density disappears, $q(|z|)\wedge q(|z+x-y|)=0$, for $x,y\in \rd$ which are sufficiently distant, i.e.\ $|x-y|\gg 1$. In this case, the second row of \eqref{basic-coup-1} disappears, and \eqref{basic-coup-1} essentially becomes \eqref{basic-coup-5}. This illustrates the advantage of the refined basic coupling \eqref{bas-1}: it applies both to finite range jumps and non-necessarily rotationally symmetric L\'evy processes.
\end{remark}

\subsection{Coupling operators for SDEs with multiplicative L\'evy noise}\label{subsec-multi}
It is possible to extend the coupling idea from the previous sections to SDEs with multiplicative L\'evy noise
\begin{equation}\label{s1m}
    d X_t
    = b(X_{t})\,dt + \sigma(X_{t-})\,dZ_t,\quad X_0=x\in \rd,
\end{equation}
where $b: \rd\to\rd$ is measurable, $\sigma:\rd\to\rd\times\rd$ is continuous, and $Z=(Z_t)_{t\geq 0}$ is a pure jump L\'{e}vy process on $\rd$   with L\'evy measure $\nu$.   Since the drift term $b$ is always coupled synchronously, we only need to consider how to couple multiplicative L\'evy noise.  The multiplicative term $\sigma(x)$ affects the jumps in a way that the jump height $\Delta Z_t = Z_t - Z_{t-}$ is not simply added to $X_{t-}$ (as in the additive noise case) but it is first transformed by the matrix $\sigma(X_{t-})$ and then added. This means that in our coupling scheme \eqref{basic-coup-4} we have to
\begin{gather*}
    \text{replace\ \ $\Psi_i(z)$\ \  by\ \  $\sigma(y)\Psi_i(z)$}.
\end{gather*}
More precisely, for any $1\leq i<n+1\leq \infty$, let $\Psi_i:\rd\to\rd$ be a bijective continuous map and $\nu_i$ a nonnegative measure on  $(\rd,\Bscr(\rd))$ such that $\sum_{i=1}^n \nu_i\leq \nu$. Now we change the general formula \eqref{basic-coup-4} for the basic coupling with additive noise to
\begin{equation}\label{basic-coup-455}
(x,y)\longmapsto
    \begin{cases}
        (x+\sigma(x)z, y+\sigma(y)\Psi_i(z)), & \mu_{\nu_i,\Psi_i}(dz)\text{\ \ for\ \ } 1\leq i< n+1;\\
        (x+\sigma(x)z, y+\sigma(y)z), & \big(\nu -\sum_{i=1}^n \mu_{\nu_i,\Psi_i} \big)(dz).
    \end{cases}
\end{equation}
As in the proof of Proposition~\ref{P:co}, we can verify that \eqref{basic-coup-455} determines a coupling operator for the infinitesimal generator of the SDE \eqref{s1m} if \eqref{e:cou} is satisfied. It is reasonable that in the case of multiplicative L\'evy noise, the maps $\Psi_i(z)$ should depend on the coefficient $\sigma(x)$. In view of the results from Sections~\ref{subse1} and~\ref{subsection2}, let us discuss the following examples.

\begin{example}[Coupling by reflection for multiplicative L\'{e}vy noise]
Assume that $Z$ is a pure jump rotationally symmetric L\'evy process with L\'evy measure $\nu$. Let $n=2$, $\nu_1(dz)=\nu_2(dz)=\frac{1}{2}\I_{\{|z|\leq \eta |x-y|\}}\,\nu(dz)$ for some $\eta\in (0,\infty]$,
\begin{gather*}
    \Psi_1(z)
    = \sigma(y)^{-1}\sigma(x) R_{x,y}(z)
 \et
    \Psi_2(z)
    = \Psi_1^{-1}(z)
    = R_{x,y}(\sigma(x)^{-1}\sigma(y) z);
\end{gather*}
where $R_{x,y}(z)$ is  the reflection operator   defined in \eqref{e:effrr}. It is easy to see from the rotational invariance of the L\'evy measure $\nu$ and the properties of $R_{x,y}(z)$, that setting $\sigma(x)=\id_{d}$ reduces  \eqref{basic-coup-455}  to \eqref{basic-coup-5}.
\end{example}

\begin{example}[Refined basic coupling for multiplicative L\'{e}vy noise]
Let $Z$ be an arbitrary pure jump L\'evy process, $n=2$ and $\nu_1=\nu_2= \nu/2$. For any $\kappa>0$ and $x,y\in \rd$ with $x\neq y$,
let
\begin{gather*}
    \Psi_1(z)
    = \Psi_{\kappa,x,y}(z)
    := \sigma(y)^{-1}\big(\sigma(x)z+(x-y)_\kappa\big)
\intertext{and}
    \Psi_2(z)
    = \Psi^{-1}_1(z)
    = \sigma(x)^{-1}\big(\sigma(y)z-(x-y)_\kappa\big).
\end{gather*}
Again, if $\sigma(x)=\id_{d}$,  \eqref{basic-coup-455}  becomes
\eqref{bas-1}. This coupling was first introduced in \cite{LW18} when studying the regularity of semigroups and the ergodicity of the solution to the SDE \eqref{s1m}.
 \end{example}

\begin{ack}
    The research of Jian Wang is supported by the National Natural Science Foundation of China (Nos.\ 11522106 and 11831014), the Fok Ying Tung Education Foundation (No.\ 151002), the Alexander von Humboldt foundation, the Program for Probability and Statistics: Theory and Application (No.\ IRTL1704) and the Program for Innovative Research Team in Science and Technology in Fujian Province University (IRTSTFJ).
\end{ack}

\end{document}